\documentclass{article}
\usepackage{amsmath,amssymb, amsthm, breqn}
\usepackage{hyperref}

\pdfoutput=1

\title{A Perfect Set of Reals with Finite Self-Information}
\author{Ian Herbert}

\begin{document}

\newtheorem{thm}{Theorem}[section]
\newtheorem{lemma}[thm]{Lemma}
\newtheorem*{thm*}{Theorem}
\newtheorem*{thmref}{Theorem \ref{perfectsetallf}}
\newtheorem{prop}[thm]{Proposition}
\newtheorem{claim}[thm]{Claim}
\newtheorem{cor}[thm]{Corollary}
\newtheorem{ex}[thm]{Example}
\newtheorem{?}[thm]{Question}
\newtheorem{alg}[thm]{Algorithm}

\theoremstyle{definition}
 \newtheorem{definition}[thm]{Definition}
  \newtheorem{example}[thm]{Example}
 \newtheorem{remark}[thm]{Remark}

\newcommand{\R}{\mathbb{R}}
\newcommand{\Q}{\mathbb{Q}}

\newcommand{\Z}{\mathbb{Z}}
\newcommand{\proj}{\mathbb{P}}
\newcommand{\C}{\mathbb{C}}
\newcommand{\V}{\mathcal{V}}
\newcommand{\A}{\mathcal{A}}
\newcommand{\f}{\overline{f}}
\newcommand{\I}{\overline{I}}
\newcommand{\1}{\mathds{1}}

\newcommand{\s}{\sigma}
\newcommand{\ua}{\upsilon^{\alpha}}
\newcommand{\ma}{\mu^{\alpha}}
\newcommand{\mb}{\mu^{\beta}}
\newcommand{\uas}{\upsilon^{\alpha_{\sigma}}}
\newcommand{\nin}{\notin}
\newcommand{\phii}{\Phi_i}
\newcommand{\twow}{2^{<\omega}}
\newcommand{\empstr}{\langle \ \rangle}
\newcommand{\emp}{\emptyset}
\newcommand{\fbar}{\textit{\underbar{f}}}
\newcommand{\fhat}{\hat{f}}
\newcommand{\abar}{\underbar{A}}
\newcommand{\rst}{\mid}
\newcommand{\conc}{}
\newcommand{\<}{\langle}
\newcommand{\lam}{\Lambda}
\newcommand{\del}{\Delta}
\newcommand{\as}{\alpha_{\sigma}}
\newcommand{\asj}{\alpha_{\sigma_j}}
\newcommand{\aspre}{\alpha_\sigma^{*}\conc\sigma(i)}
\newcommand{\mpr}{M_p^r}
\newcommand{\dom}{\text{dom}}
\newcommand{\scripto}{\mathcal{O}}
\newcommand{\restrict}{\upharpoonright}
\newcommand{\U}{\mathbb{U}}
\newcommand{\Dim}{\text{Dim}}
\newcommand{\phe}{\phi_e}
\newcommand{\phes}{\phi_{e,s}}
\newcommand{\phat}{\hat{\phi}}
\newcommand{\nais}{n^{\alpha}_{i,s}}
\newcommand{\phates}{\hat{\phi}_{e,s}}

\maketitle
\begin{abstract} We examine a definition of the mutual information of two reals proposed by Levin in \cite{levin}. The mutual information is 
 $$I(A:B)=\log\sum\limits_{\sigma,\tau\in\twow} 2^{K(\s)-K^A(\s)+K(\tau)-K^B(\tau)-K(\s,\tau)},$$
 where $K(\cdot)$ is the prefix-free Kolmogorov complexity. A real $A$ is said to have finite self-information if $I(A:A)$ is finite. We give a construction for a perfect $\Pi^0_1$ class of reals with this property, which settles some open questions posed by Hirschfeldt and Weber. The construction produces a perfect set of reals with $K(\s)\leq^+ K^{A}(\s)+f(\s)$ for any given $\Delta^0_2$ $f$ with a particularly nice approximation and for a specific choice of $f$ it can also be used to produce a perfect $\Pi^0_1$ set of reals that are low for effective Hausdorff dimension and effective packing dimension. The construction can be further adapted to produce a single perfect set of reals that satisfy $K(\s) \leq^+ K^A(\s)+f(\s)$ for all $f$ in a `nice' class of $\Delta^0_2$ functions which includes all recursive orders. 
\end{abstract}

\section{Mutual Information}

There has been some interest in recent years in developing in a recursion-theoretic framework a formal treatment of the concept of the information content of a reals (infinite strings of 0's and 1's) and in particular what it means for two reals to share information. Turing reductions capture one concept of information content: information about other reals. In this understanding of information a real $A$ `knows about' a real $B$ just in case $A\geq_T B$ and the mutual information of two reals then is just the Turing ideal generated by the pair. These ideals have been well-studied in their own right. However, one could instead be interested in a different kind of information, namely, information about finite sets (finite strings of 0's and 1's). Every finite set is recursive, so here Turing reduction is not a useful concept. Since even the empty set `knows' about all these finite sets, the interesting quantity here is how much more a real $A$ knows about a string $\s$ than the empty set does. The prefix-free complexity of a string $\s$, denoted $K(\s)$ is the length of the shortest description of $\s$ for some universal prefix-free decoding machine $\mathbb{U}$ (see either \cite{downhirsch} or \cite{nies} for a more in-depth discussion of Kolmogorov complexity) and is often used as a measure of the informational content of $\s$. This notion relativizes easily by letting the universal decoding machine have access to an oracle for its computations, and then the difference $K(\s)-K^{A}(\s)$ is a measure of how much better $A$ is at describing $\s$ than the empty set. Using this as as formalization of our intuitive notion of information we can define the mutual information of two reals as follows. 

\begin{definition}\label{mutinfo} The \textit{mutual information} of two reals $A$ and $B$ is\\
\\
$$I(A:B)=\log{\sum\limits_{\sigma , \tau \in 2^{<\omega}}2^{K(\s)-K^A(\s)+K(\tau)-K^B(\tau)-K(\sigma , \tau)}}.$$\end{definition}
As above, here $K(\s)-K^A(\s)$ is some rough idea of how much $A$ `knows' about $\s$, and  $K(\tau)-K^B(\tau)$ is how much $B$ `knows' about $\tau$. We weight the summand by $2^{-K(\s,\tau)}$, where $K(\s,\tau)$ is the complexity of the pair $(\s, \tau)$, so that pairs of strings that are more closely related contribute more to the sum (it is easier to describe the pair $(\s,\tau)$ if $\s$ and $\tau$ are closely related, so the complexity will be lower). This ensures that reals $A$ and $B$ will have high mutual information if and only if the strings that they each know a lot about are related enough. The distinction we will be interested in for different pairs of reals is whether $I(A:B)$ is finite or infinite.
Definition~\ref{mutinfo} is equivalent to a definition proposed by Levin in \cite{levin}, and is the one used by Hirschfeldt and Weber in \cite{hirschfeldtweber}. In the same paper, Levin proposed another defintion of mutual information that has since been called \emph{simplified mutual information}, where the sum in Definition~\ref{mutinfo} is only over pairs with $\s=\tau$. Clearly, if the mutual information of two reals is finite then so is the simplified mutual information, but it is open whether the converse holds, i.e., whether the notions coincide.  

 The main result of the Hirschfeldt and Weber paper involved the notion of finite self-information.
 \begin{definition} A real $A$ \textit{has finite self-information} if $I(A:A)<\infty$.\end{definition}

Levin had posed the question as to whether this notion coincided with that of $K$-triviality (a real $A$ is \emph{K-trivial} if there is a $b$ such that, for all $n$, $K(A\restrict_n)=K(n)+b$  or, equivalently, if there is a $c$ such that for all $\s$ $K^{A}(\s) \geq K(\s)-c$. This is called being \emph{low for K}. The equivalence is due to Nies \cite{lowforkktrivial} ).  Clearly any $K$-trivial real has finite self-information, since the difference $K(\s)-K^{A}(\s)$ is at most $c$ for all $\s$, and so $I(A:A)$ is bounded by $2c + \log{\sum\limits_{\sigma , \tau \in 2^{<\omega}}2^{-K(\sigma , \tau)}}$, which is finite. Hirschfeldt and Weber showed that the converse fails.

\begin{thm}\emph{(Hirschfeldt, Weber \cite{hirschfeldtweber} )}  There is a real that has finite self-information that is not $K$-trivial.\end{thm}
There remained the question of whether the set of reals with finite self-information was `like' the set of $K$-trivials. In particular, Hirschfeldt and Weber asked whether there were only countable many reals with finite self-information and whether any such real had to be $\Delta^0_2$, two properties that hold of the set of $K$-trivials. This paper answers both of these questions in the negative, by constructing a perfect set of reals that have finite self-information.

\begin{thm}\label{lowforinfoperfect} There is a perfect set of reals that have finite self-information.\end{thm}

The proof of Theorem~\ref{lowforinfoperfect} builds on the techniques of Hirschfeldt and Weber, in particular the following lemma and corollary.
\begin{lemma} \label{fexists}\emph{(Hirschfeldt, Weber \cite{hirschfeldtweber} )} There is a function $f$ such that \linebreak $\sum\limits_{\s , \tau}2^{-K(\s, \tau)+f(\s)+f(\tau)}$ converges and $f$ has a computable approximation such that $(\forall i)(\forall^{\infty}\s)(\forall s)[f_s(\s)>i].$
\end{lemma}

\begin{cor}\label{hirschwebcor} For $f$ as in the lemma, if $K(\s)\leq^{+} K^{A}(\s)+f(\s)$ for all $\s \in 2^{<\omega}$, then A has finite self-information.\end{cor}

Hirschfeldt and Weber prove Lemma~\ref{fexists} by building such an $f$ explicitly. The construction is interesting but rather involved, and we will not need any properties of their $f$ beyond those mentioned in the Lemma. The corollary follows easily by using the given bound and the definition of finite self-information. This gives a sufficient condition for a real $A$ to have finite self-information, but it is still open whether some condition like $K(\s)\leq^{+} K^{A}(\s)+f(\s)$ for all $\s \in 2^{<\omega}$ is necessary. 

To prove Theorem~\ref{lowforinfoperfect}, we prove a more general theorem, of which it will be a corollary.

\begin{thm}\label{perfectsetf} If $f:\twow \rightarrow \mathbb{N}$ is total and has a recursive approximation $(f_s)$ with $\forall i\forall^{\infty}\s\forall s f_s(\s)>i$, then there exists a perfect $\Pi^0_1$ set $\mathcal{P}$ and a constant $c$ such that for any $A\in \mathcal{P}$ and any $\s \in \twow$, $K(\s)\leq K^A(\s)+f(\s) +c$. Moreover, for any real $C$, there exist $A,\ B \in \mathcal{P}$ such that $C \leq_{T} A\oplus B$.
\end{thm}

It is clear that Theorem~\ref{lowforinfoperfect} follows from Theorem~\ref{perfectsetf} and Corollary~\ref{hirschwebcor}. Note that Theorem~\ref{perfectsetf} is saying that, for some understanding of `reasonable,' any reasonable weakening of lowness for $K$ has an uncountable class of witnesses. There are only countably many reals that are low for $K$, so being low for $K$ is in some sense a `maximally weak' lowness notion with only countably many witnesses. We prove Theorem~\ref{perfectsetf} by building a recursive tree $T\subset \twow$ with $2^{\omega}$-many infinite paths, none of which is isolated, and each of which satisfies the condition $K(\s)\leq^+ K^A(\s)+f(\s)$ for all $\s$ and the given $f$. To ensure that for any real $C$, we have a pair of paths that join above $C$, we will force all of our paths to be identical except for designated `coding locations,' and then the join of the paths with $C$ and $\bar{C}$ in their coding locations will be able to compute $C$.

Finally, we will extend the method of Theorem~\ref{perfectsetf} to build a perfect tree such that its paths satisfy $\forall \s \ K(\s)\leq^{+} K^A(\s)+f(\s)$ for all such functions. This tree will be a subtree of a recursive tree, but picking it out will be arithmetically more complicated. 

\begin{thm}\label{perfectsetallf} There is a perfect set $\mathcal{P}$ such that for any total $f: \twow \rightarrow \mathbb{N}$ with a recursive approximation $(f_s)$ such that $\forall i\forall^{\infty}\s\forall s f_s(\s)>i$, there is a $c_f$ such that for any $A\in \mathcal{P}$ and any $\s \in \twow$, $K(\s) \leq K^A(\s)+f(\s)+c_f$. Moreover, for any real $C$, there exist $A,\ B \in \mathcal{P}$ such that $C \leq_{T} A\oplus B$.\end{thm}

Theorem~\ref{perfectsetallf} will be proved in Section 6.

\section{Building the Tree}

The idea behind the proof of Theorem~\ref{perfectsetf} is to build the a tree $T$ with a finite injury priority construction to satisfy the requirements
$$R_i: T\text{ has at least } i+1 \text{-many levels where all living nodes branch}$$  
and
$$S_{i}: \ \  \text{For all } \s \in \twow \text{ with } f(\s)=i, K(\s)\leq^{+} K^{A}(\s)+f(\s) \text{ for all } A \in [T]$$ 
for all $i \in \omega$.

The set of paths through $T$, $[T]$, will be the perfect set we are looking for. 

The strategy for meeting the $R_i$ requirements is simple. When a requirement first becomes active and every time it is injured it picks a new number  bigger than any number mentioned so far in the construction to be $n_i$, extends each still living branch of the tree with a string of 0s to height $n_i$, and then extends each of these in both directions. This doubles the number of currently living branches, and if $R_i$ acts after $R_j$ for all $j < i$ and none of these are ever injured again, there will be at least $2^{i+1}$-many infinite paths through $T$

The strategy for meeting an $S_i$ requirement is also simple, but there is added complexity when we try to run them all together. In addition to building $T$, we build a Kraft-Chaitin set $L$ (for more on Kraft-Chaitin sets, see \cite{downhirsch} or \cite{nies}). The strategy for $S_i$ monitors the enumeration of the universal decoding machine $\U$ relative to the partial paths through $T$ that are still alive. If it sees that for some $\s$ with $f_s(\s)=i$ (so $\s$ is  a string that $S_i$ is concerned with) and some living partial path $\alpha\in T_s$ that $K^{\alpha}_s(\s)+f_s(\s)$ is less than the current best description of $\s$ that it has put into $L$, it waits for its turn to act and then asks to put the pair $\langle \s, K_s^{\alpha}(\s)+f_s(\s) \rangle$ into $L$ for the partial path $\alpha$ that minimizes $K_s^{\alpha}(\s)$. The larger construction must decide whether or not to allow this, based on how much mass this would put into $L$. 

The problem, of course, is that in general the $L$ created by all these strategies running together is \textit{not} a Kraft-Chaitin set, since we have no \textit{a priori} bound on the amount of mass we are using for requests, that is, on $\sum\limits_{\langle \s , l\rangle \in L}2^{-l}$. We are building infinitely many paths $A$ through $T$, and potentially each one could have a short description of some $\s$. Then we would have to put arbitrarily short descriptions of all $\s$'s into L, and this could push the measure arbitrarily high. In order for the Kraft-Chaitin Theorem to give us a machine $M_L$ with $K_{M_L}(\s)\leq\min\{l:\langle \s , l\rangle \in L\}$ for all $\s$, the measure must be no more than 1. As long as we can get any finite bound on the measure we can make all the lengths of descriptions longer by a constant and absorb the new constant into the constant of the machine's index. The work of the construction will be tohandle requests from the strategies for the $S_i$'s to add to $L$ without letting the measure grow too large.  

We have two advantages that will allow us to do this. First, although we are concerned about descriptions using $2^{\omega}$-many different sets as oracles, we are building our sets as paths on a tree, so they will have stems in common among them.  We know that the total mass that $\U$ can use along any path is bounded by 1 and, following conventions on use of the computations of $\U$ (namely, that if $\U_s^{\alpha}(\tau)\downarrow=\s$ then $\U_t^{\beta}(\tau)\downarrow=\s$ for all $\beta \succeq \alpha$ and all $t\geq s$) we can push the branchings of the tree to heights below which some amount of mass has already converged. This gives us some control over how many new descriptions we'll have to deal with and how short they can be. Second, we only have to match short descriptions of $\s$ up to a factor of $f(\s)$. We know $\liminf f =\infty$, so we can keep the measure of $\dom(M_{L})$ down by only allowing those $\s$ with large $f$ values to have descriptions that appear on nodes higher than the first few branchings of the tree. The goal will be (very roughly) to ensure, for each $\s$, if $f(\s)=i$ then all descriptions of $\s$ appearing on any path through T appear on initial segments with fewer than $i$ branching nodes. 
\\
\\
\section{The Construction}
Before we give the construction we formalize some terminology.
\\
\\
To \textit{kill} a node in our partially constructed tree is to make a commitment to never put nodes above it in our tree. After we kill a node, it is \textit{dead}; before it is killed it is \textit{living}.
\\
\\
As long as a number $n_i$ is associated to some $R_i$ strategy  it is called a \textit{branching level}, and nodes at that height, \textit{branching nodes}. When there is an injury and the value of $n_i$ changes, the old value is no longer a branching level, and the old nodes are no longer branching nodes. The coding locations will be the numbers $n_{i}+1$ for the $n_i$ that the construction settles on. We will write $n_{i,s}$ for the value of $n_i$ at stage $s$ 
\\
\\
To simplify the construction and the proofs, we will modify the function $f$ slightly. We would like our function to only take certain values, so that the proofs that the measure of $\dom(M_{L})$ is bounded will go more smoothly. With this is mind, we define the sequence $\{c_i\}_{i\in \omega}$ to be $c_0=0$ and $c_i=4^i$ for $i>0$. Now we let $\fhat_s(\s)= \text{ the least } c_i \text{ such that } f_t(\s)<c_{i+1} \text {for some } t\leq s$. 

Clearly, we still have that $(\forall i)(\forall^{\infty} \s)(\forall s) [\fhat_s(\s)>i]$, since all that has happened is that all the $\s$'s that at any stage were sent to values less than 4 (only finitely many, by the same property for $f$) are now sent to $0$, those that were ever sent to values less than 15, but never less than 4, are now sent to 4, etc. Also, $\fhat_s(\s)\leq f_s(\s)$ for all $\s$ and $s$, so $K(\s)\leq ^{+} K^{A}(\s)+\fhat(\s)$ implies $K(\s)\leq^{+} K^{A}(\s)+f(\s)$. Another nice feature is that $\fhat$ is monotone non-increasing in $s$, so we won't have to worry in the construction about these values going up. We redefine our $S_i$ requirements to account for this change:

$$S_{i}: \ \  \text{For all } \s \in \twow \text{ with } \fhat(\s)=c_i, K(\s)\leq^{+} K^{A}(\s)+\fhat(\s) \text{ for all } A \in [T]$$ 

We will say that $S_i$ \textit{has control of} $\s$ at stage $s$ if $\fhat_s(\s)=c_i$. 

An $S_i$ strategy \textit{requires attention} at a stage $s$ if there are a $\s$ that $S_i$ has control of and a living node $\alpha$ in $T_s$ and $K_{s}^{\alpha}(\s)+\fhat_{s}(\s)$ is less than the shortest description of $\s$ in $L_s$. This means that we are not guaranteeing the inequality $K(\s) \leq K^{A}(\s) + \fhat(\s)$ for any $A \in [T]$ that extends $\alpha$.

An $R_i$ strategy \textit{requires attention} at a stage $s$ if it does not have a number $n_i$ associated to it.  
\\
\\
Now we give the \textbf{Construction}:
\\
\noindent Order the Requirements $S_0, R_0, S_1, R_1, \hdots$.\\

\noindent \textbf{Stage 0}: Set $T_0=\emp$,  $L_0=\emp$
\\
\\
\textbf{Stage $s$+1}: \\
\textbf{Substage 1}: Calculate $\fhat_{s+1}(\s)$ and $K_{s+1}^{\alpha}(\s)$ for all living branches $\alpha$ in $T_s$ and the first $s+1$ $\s$'s. 
\\
\\
For all $\s$ such that $\fhat_{s+1}(\s) \neq \fhat_s(\s)$ or that first took an $\fhat$ value at this stage $S_i$ gains control of $\s$ where $\fhat_{s+1}(\s)=c_i$ and any other $S_j$ loses control of $\s$. Go to Substage 2.
\\
\\
\noindent\textbf{Substage 2}: If one of the first $s+1$ requirements requires attention, do the following for the one with lowest (closest to 0) priority:
\\
\\
\textbf{Case 1} If it is an $R_i$ requirement, then:\\
\indent 1.) Pick a new $n_i$ bigger than anything seen yet in the construction\\
\indent 2.) Let $T_{s+1}=T_s \cup \{ \alpha \conc \beta \conc j  | \alpha \text{ is a living leaf node in } T_s, \ |\alpha \conc \beta|=n_i, \ j=0 \text{ or } 1, \text{ and } \beta(k)=0 \text{ for all }k \text{ where it is defined} \}$\\
\indent 3.) $n_i$ is now associated to $R_i$ and is a branching level.\\
\indent 4.) $L_{s+1}=L_s$\\
\\
\\
\textbf{Case 2} If it is an $S_i$ requirement then since it requires attention, for some partial path $\alpha$ through $T_s$ and some $\s$ that $S_i$ has control over, $K_{s+1}^{\alpha}(\s)+\fhat_{s+1}(\s)<\min\{l | \langle \s, l \rangle \in L_s\}$. There are two subcases, depending on the use $\ua_{s+1} (\tau)$ of the computation $\U_{s+1}^{\alpha}(\tau)=\s$ giving the new shorter description for the lexicographically least $\s$ that this holds for.
\\
\\
\textbf{Subcase 1}: $\ua_{s+1}(\tau)\leq n_i$ (or $n_i$ not currently defined) where $\fhat_{s+1}(\s)=c_i$ 
\\
This is fine. $n_i$ is the level where we branch for the $(i+1)$st time, so if our new description appears before that level, we can make the adjustments to keep up with this change.\\
\indent 1.) Let $T_{s+1}=T_{s}$.\\
\indent 2.) Put a new request $\langle \s , \  K_{s+1}^{\alpha}(\s)+\fhat_{s+1}(\s)\rangle$ into $L_{s}$ to get $L_{s+1}$.\\
\\
\\
\textbf{Subcase 2}: $\ua_{s+1}(\tau)> n_i$ where $\fhat_{s+1}(\s)=c_i$ 
\\
This is a problem. Since $\fhat_{s+1}(\s)=c_i$, we want to only have descriptions of $\s$ appearing before the tree branches $i+1$ many times, but we have some living node $\alpha$ on $T$ with length greater than $n_i$ which gives a new shorter description of $\s$. To correct this: \\
\indent 1.) Injure $R_i$ and run the Injury Subroutine on it.\\ 
\indent 2.) Let $T_{s+1}=T_s$ and $L_{s+1}=L_s$.\\
\\
\textbf{Injury Subroutine} for $R_i$:\\
\indent 1.) Find the node $\alpha$ of $T_s$ at level $n_{i}$ and the string $\gamma$ such that $\alpha \conc \gamma$ is a living leaf node of $T_s$ and that maximizes $\sum\limits_{\tau: \ \U_{s+1}^{\gamma}(\tau)\downarrow, \ \U_{s+1}^{\alpha}(\tau)\uparrow} 2^{-|\tau|}$. If there is more than one such pair, choose the leftmost.\\
\indent 2.) For all living nodes $\beta$ at level $n_{i}$ keep the leaf node $\beta \conc \gamma$ alive; kill all other nodes above $\beta$. Set all $R_k$ for $k\geq i$ to requiring attention (i.e. disassociate from each $R_k$ its $n_k$). \\
\\
This ends the construction. We let $T=\bigcup\limits_{s}T_s$ and $L=\bigcup\limits_{s}L_s$. We now verify that the construction works.

\begin{lemma}\label{finiteinjury} Each $R_i$ is injured only finitely often.\end{lemma}

\begin{proof} Each $R_i$ has only finitely many $S_i$'s before it in the ordering of requirements, and by our condition on $f$, each of these only ever gets control of finitely many $\s$'s. Now we need to prove that each such $\s$ can only be the cause of finitely many injuries. Suppose some $R_i$ requirements are injured infinitely often, and let $R_j$ be the least in the ordering of requirements. Since only finitely many $\s$'s can injure $R_j$, at least one must do so infinitely often. Let $\s'$ be the lexicographically first.

By our assumption on $j$, each $R_k$ with $k<j$ is injured only finitely often and we can assume we are at a stage after this has happened for the last time. Let us also assume we are also at a stage $s>t$, where $\fhat_t(\s')$ has converged to its final value $c_j$ ($\fhat(\s')$ must be $c_j$, or else there would be an earlier $R_k$ that would also be injured infinitely often). The Injury Subroutine on $R_j$ will keep only $2^j$ many leaf nodes alive above the level $n_j$ ($n_j$ will change each time the requirement $R_j$ acts, there will always be $2^j$ many nodes at level $n_j$). Now any injury caused by $\s'$ happens because we find  a description, $\tau$, of $\s'$ that is shorter than the description we have in $L_s$ and that appears on some initial segment of height at least $n_j$. That means that at least $m:=2^{|\tau|}$ has converged above $n_j$, so the Injury Subroutine will pick a $\gamma$ with at least that much new mass. Since the $\gamma$ picked by the Injury Subroutine is kept alive, one of the $2^j$ branches that are left will have at least $m$ more mass than its initial segment up to level $n_{j,s}$. 

Each run of the Injury Subroutine will either keep the path with the shorter description of $\s'$ alive or kill it. If it is kept alive, at the next stage it will cause $S_j$ to still require attention, but the use will now be low enough (below $n_j$) that the new request will be added to $L$. In this case, the minimum length of a description of $\s'$ that would give a cause to injure $R_j$ drops by 1, so clearly this can only happen finitely often. After this stage, the amount guaranteed to be gained by some branch below $n_j$ is bounded below by some $m':=2^k$, where $k$ is the length of the last description of $\s'$ that was accepted minus 1. 

Each injury to $R_j$ by $\s'$ now adds at least $m'$ to one of $2^j$ many branches below $n_{j,s}$. We are assuming no $R$ requirements with lower priority will ever be injured again, so once a path is chosen by the Injury Subroutine for $R_j$ it stays a path through the living subtree for the rest of the construction. Thus, the mass on the $2^j$ initial branches of the living tree increases by at least $m'$ for every injury that $\s'$ makes to $R_j$, and since the total mass that can converge on any string is bounded by 1 (it will be the measure of the domain of the universal prefix-free machine relative to that string), there can be no more than $\frac{1}{m'}\cdot 2^j$ many subsequent injuries to $R_j$ from $\s'$. Thus, after all injuries to $R_k$ for $k<j$ have stopped, there are only finitely more to $R_j$, so each $R_i$ can be injured only finitely often.   

\end{proof} 
\begin{lemma} Each requirement is eventually satisfied\end{lemma}
\begin{proof} By Lemma~\ref{finiteinjury} each requirement is injured only finitely often. Each $R_i$ only needs to act once as long as it is never injured again. Each $S_i$ can only act finitely many times after it has finished injuring $R$ requirements until it's satisfied, since there is a lower bound on the length of possible descriptions of the finitely-many $\s$'s it has control over. So any requirement will eventually no longer be injured and be the lowest priority argument that ever requires attention again. Then when it requires attention it will act finitely often and be satisfied. \end{proof} 

\section{Bounding the Domain}
We have shown that each requirement is eventually satisfied, but in satisfying the $S_i$ requirements we may have put too much mass into our request set $L$, so that the Kraft-Chatin theorem no longer applies. We now need to find an \textit{a priori} bound on the mass put into $L$, i.e., on $\Lambda := \sum\limits_{\langle \s, \ l\rangle \in L}2^{-l}$. As long as we can get some \textit{a priori} bound, we know there is a Kraft-Chaitin set where each description is longer by just a constant, so we satisfy the property in Corollary~\ref{hirschwebcor}.

First, we define a subtree 
$$T'=\{\alpha\in T| \alpha \text{ is not killed in any run of the Injury Subroutine}\}.$$ 
$T'$ is the living subtree of $T$. Since any $\langle\s , \ l\rangle$ goes into $L$ only when a new description of $\s$ converges on some living node of $T$, we may divide $L$ into two parts:\\
\begin{multline*}
L '=\{\langle\s , \ l\rangle \in L| \langle\s , \ l\rangle \text{ was put into } L \text{ in response to a description} \\ \text{converging on an initial segment in } T'\}
\end{multline*}
\begin{multline*}
L ''=\{\langle\s , \ l\rangle \in L| \langle\s , \ l\rangle \text{ was put into } L \text{ in response to a description} \\ \text{converging on an initial segment not in } T'\}.
\end{multline*}

We will bound the measures of the domains of these two sets separately. 

To simplify the proof we may assume that any string $\tau$ such that $\U_s^{\alpha}(\tau)$ converges for some living $\alpha\in T_s$ is a shorter description of some $\sigma$ we are monitoring (i.e. any $\tau$ for which $\U_s^{\alpha}(\tau)\downarrow$ causes us to either put a new request $\langle \U_s^{\alpha}(\tau), \ |\tau|+\fhat_s(\U_s^{\alpha}(\tau))\rangle$ into $L_s$ or causes an injury). 
Clearly this just increases the measure of $\dom(M_{L})$. We call a pair $(\alpha , \tau)$ \textit{exact} if $\U^{\alpha }(\tau)\downarrow$ and $|\alpha |=\ua (\tau)$, the use of the computation. Thus, we bound $\lam$ by instead bounding 
$$\del=\sum\limits_{(\alpha,\tau) \text{ exact}: \alpha \in T \text{ and } \alpha \text{ is alive at a stage when } \U_s^{\alpha}(\tau)\downarrow}2^{-|\tau|-\fhat(\U^{\alpha}(\tau))}\cdot 2.$$
$\ $\\
We add an extra factor of 2 to account for the fact that $\fhat$ is slowly converging to its final value, so we may put requests into $L$ for finitely many $\fhat_s$ values larger than the final one. This has the effect of at most doubling the mass in $L$. Clearly $\lam < \del$, and in particular $\lam '<\del '$, where $\lam'$ is defined with $L'$ instead of $L$ and $\del '$ is defined with $T'$ instead of $T$

Now we can find a bound for $\del '$.

\begin{lemma}\label{del'}$\del ' \leq 2$\end{lemma}
\begin{proof}For each $\s \in \twow$, let $\alpha_{\s}$ be the node in $T'$ such that $|\alpha_{\s}|=n_{|\s|}$ and for all $i < |\s|$, $\alpha_{\s}(n_i+1)=\s(i)$. That is, $\alpha_{\s}$ is the partial path through $T'$ that follows $\s$ at the branching levels.

Now define $Q_{\s}=\{\tau \in \twow|\U^{\alpha_{\s}}(\tau)\downarrow \text{ and } \U^{\alpha_{\s'}}(\tau)\uparrow \text{ for } \s' \prec  \s\}$. Note that each exact pair $(\alpha, \tau)$ contributes exactly one $\tau$ to exactly one $Q_\s$ (for $\s$ least such that $\alpha \prec \alpha_{\s}$).

Finally, define $m_{\s}=\sum\limits_{\tau \in Q_{\s}}2^{-|\tau|}$. $m_{\s}$ is the measure of the strings $\tau$ that converge as descriptions on $T'$ along the path to $\alpha_{\s}$ and with uses $\uas(\tau)>n_{|\s|-1}$. Since we are concerned now only with $T'$ there can be no injuries to the branches we have. Thus $\fhat(\U^{\alpha_{\s}}(\tau))$ must be at least $c_{|\s|}$ for $\tau \in Q_{\s}$ or else there would be an injury to $R_{|\s|-1}$ and our $T'$ would change. By definition $c_i \geq i$ for all $i$, so we have that $\fhat(\U^{\alpha_{\s}}(\tau))\geq |\s|$ for $\tau \in Q_{\s}$.
So, for each exact pair $(\alpha,\tau)$ we put at most $2^{-|\tau|-|\s|+1}$ much mass into $\del'$, where $\s$ is the shortest string with $\alpha \prec \alpha_{\s}$. 

Thus, the total amount we put in for all $\tau$'s in $Q_{\s}$ is $\frac{2\cdot m_{\s}}{2^{|\s|}}$. Summing over all $\s$, we get the total amount of mass in $\del '$ is bounded by $\sum\limits_{\s}\frac{2 \cdot m_{\s}}{2^{|\s|}}$. 

Now, it is clear that any $m_{\s}$ must be less than 1, since it is the measure of some subset of the domain of the universal machine relative to some oracle (namely, $\alpha_{\s}$). However, since $\alpha_{\s}$ extends $\alpha_{\s'}$ for $\s\succ \s'$, it is also true that for any $\s$ $\sum\limits_{\s' \preceq  \s}m_{\s'}\leq 1$. Now, $\sum\limits_{\s}\frac{2 \cdot m_{\s}}{2^{|\s|}}$ is the limit as $n \rightarrow  \infty$ of the partial sums $\sum\limits_{|\s| \leq n}\frac{2 \cdot m_{\s}}{2^{|\s|}}$. For any $\s '$ of length less than $n$ there are $2^{n-|\s '|}$-many $\s$'s of length $n$ with $\s ' \prec \s$, so the sum
$\sum\limits_{|\s|=n}\sum\limits_{\s'\preceq\s}\frac{2\cdot m_{\s'}}{2^{|\s '|}}$ counts the term $\frac{2\cdot m_{\s'}}{2^{|\s '|}}$ $2^{n-|\s'|}$-many times for each $\s'$ of length less than $n$. To count this term just once for each $\s'$ we divide by $2^{n-|\s'|}$. 
This gives us that 
$$\sum\limits_{|\s| \leq n}\frac{2 \cdot m_{\s}}{2^{|\s|}}=\sum\limits_{|\s|=n}\sum\limits_{\s'\preceq\s}\frac{2\cdot m_{\s'}}{2^{|\s '|}}\frac{1}{2^{n-|\s'|}}=\sum\limits_{|\s|=n}\sum\limits_{\s'\preceq\s}\frac{2\cdot m_{\s'}}{2^{n}}=\sum\limits_{|\s|=n}\frac{2}{2^n}\sum\limits_{\s'\preceq\s}\cdot m_{\s'}.$$ The inner sum here is bounded by 1, and there are $2^n$ $\s$'s of length $n$, so this whole sum is bounded by 2. Since the partial sums are all bounded above by 2, $\sum\limits_{\s}\frac{2\cdot m_{\s}}{2^{|\s|}}$ must also be bounded by $2$, and so this is a bound on $\del '$.\end{proof}

All that remains is to find a bound for $\del ''=\del-\del '$. To do this we must first prove a claim about how the mass wasted in an injury relates to the mass saved on the chosen path.

\begin{lemma}\label{injurybound} For any injury to requirement $R_i$ at stage $s$ in the construction, the amount that is paid into $\del$ on the paths above $n_{i,s}$ (those kept and those killed) is no more than $\frac{1}{2}\cdot\frac{1}{2^{c_i}}$ times the mass, $m$, that converges on the path chosen by the run of Injury Subroutine for this injury.\end{lemma}

\begin{proof}
Suppose we are running the Injury Subroutine for $R_i$ at stage $s$. Then we will pick a node $\alpha$ at $n_i$ and a path $\gamma$ such that $\alpha \conc \gamma$ is a living leaf node in $T_s$ and $m=\sum\limits_{\tau: \ \U_{s}^{\alpha \conc \gamma}(\tau)\downarrow, \U_{s}^{\alpha}(\tau)\uparrow}2^{-|\tau|}$ is maximal, keep $\beta \conc \gamma$ alive for all living $\beta$ at level $n_i$, and kill all other nodes above these $\beta$. The claim is that
$$ m \cdot 2^{-c_{i}-1}\geq\sum\limits_{\eta \text{ a living leaf node}}\sum\limits_{\tau: \U_s^{\eta}(\tau)\downarrow,\ \U_s^{\eta\restrict n_i}(\tau)\uparrow}2^{-|\tau|-\fhat_s(\U_s^{\eta}(\tau))}\cdot 2.$$

The extra factor of $2$ at the end is again to account for the possibility of putting requests into $L$ for the same $\tau$ while the $\fhat$ value of its image is converging.

Before stage $s$ there was some amount of tree  built above $n_i$, say with $k$ many branchings above and including $n_i$. This gives $2^{k}$ many possibilities for $\gamma$ each divided into $k$-many segments separated by the branching nodes. To achieve an upper bound, we may assume that at worst $m$ has converged on each segment of each path. This is a gross overestimate, but for each segment of each path, $m$ \textit{is} an upper bound on the total mass possible (by the choice of $m$ by the Injury Subroutine). Call the \textit{rank} of a segment the number of branchings between that segment and level $n_i$, including the branching at $n_i$. Before the injury there are $2^{i+l}$ many living segments of rank $l$ on $T_s$ above height $n_i$: $2^l$ above each living node at height $n_i$ for $2^i$ many nodes at height $n_i$. Now, for any mass that converges on a segment of rank $l$, the $\fhat$ values of the images of the $\s$'s contributing to the mass must be at least $c_{i+l}$, or they would have caused an earlier injury. This means that if we assume $m$ much mass is converging on each segment, we can bound the right hand side of our inequality above with 
 $$\sum\limits_{l=1}^{k}\frac{m\cdot(2^{i+l})}{2^{c_{i+l}}}\cdot2.$$
 
The upper bound we are trying to get is $\frac{m}{2^{c_i}+1}$ so we want
 $$\frac{m}{2^{c_{i}+1}}\geq\sum\limits_{l=1}^{k}\frac{m\cdot(2^{i+l})}{2^{c_{i+l}}}\cdot2.$$
 $\ $\\
 Since we have the same $m$ on both sides, really it suffices for us to show that
 $$\frac{1}{2^{c_{i}+1}}\geq \sum\limits_{l=1}^{k}\frac{2^{i+l+1}}{2^{c_{i+l}}}$$ or, equivalently, that
 $$1\geq \sum\limits_{l=1}^{k}\frac{2^{c_i+i+l+2}}{2^{c_{i+l}}}.$$ 
 $\ $\\
 We know from calculus that $1 \geq \sum\limits_{l=1}^k\frac{1}{2^l}$, so we now need only show that 
$$\frac{1}{2^l}\geq\frac{1}{2^{c_{i+l}-c_{i}-i-l-2}}.$$
This reduces to showing that $\forall i \geq 0\forall l\geq1[ c_{i+l}\geq c_{i}+i+2l+2]$. Recall that $c_0=0$ and $c_i=4^i$ for all $i\geq1$ and then this also follows from elementary calculus.

Thus, the amount of mass paid into $\del$ that is affected by any injury to $R_i$ (and in particular the amount that is wasted because it was spent on paths killed by the injury) is no more than $\frac{1}{2^{c_i+1}}$ times the mass that converged on the path that was preserved.
\end{proof}

Now we can find an upper bound for $\del ''$
\begin{lemma}\label{del''} $\del '' \leq 2$ \end{lemma}
\begin{proof}
Any mass in $\del''$ is there because it was put into $\del$ when some description converged with some path as an oracle, and later the path was killed by a run of the Injury Subroutine. Thus, we can bound the total amount wasted by keeping track of the amount wasted for injuries to each $R_i$. 

An injury to $R_i$ at stage $s$ chooses the node $\alpha$ at level $n_{i,s}$ and a path $\gamma$ above $\alpha$ such that the amount of mass that converges on $\gamma$ is maximized. In doing so it fixes at least that much mass on one of the $2^i$ many paths of the living tree below the new $n_{i,t}$ (there are always $2^i$ living paths below $n_i$)

Consider the case $i=0$, when we are settling the stem of the tree. Since there are no requirements of lower priority that will be injured (just some $S_i$ requirements), each injury to $R_0$ settles some extension of the previous stem to be the new stem. From the Lemma~\ref{finiteinjury} there are only finitely many, say $k$, injuries to $R_0$. Let $m$ be the amount of mass that converges along the final stem, and $m_j$ be the mass that converges on the segment that is chosen by the $j$th injury to $R_0$. Then $m\geq\sum\limits_{j=1}^{k}m_j$ (there may be some amount of mass that converges on the stem that is added after $R_0$ has stopped being injured, but this is accounted for in $\del'$). From the lemma above, we know the total contribution to $\del ''$ from branches that are killed by these injuries to $R_0$ is no more than $\sum\limits_{j=1}^{k}\frac{m_j}{2^{c_0+1}}\leq\frac{m}{2}\leq\frac{1}{2}$.

For $i>0$ injuries to $R$ requirements with lower priority can interfere while $R_i$ is trying to fix mass to branches below $n_i$. For example, consider $R_1$. There will always be 2 branches below $n_1$ and each injury to $R_1$ puts some amount of mass on at least one of the two. However, a subsequent injury to $R_0$ will choose only one of these branches to keep alive and kill the other. Any mass that was paid into $\del$ for the branch that was killed can be charged to $R_0$, since it was its injury that killed it, but mass also went into $\del''$ for injuries to $R_1$ as mass was being put on the branch that would eventually be killed, and this mass must be charged to $R_1$. Looking at what happens with $R_1$ between the $j-1$th and $j$th injuries to $R_0$, we see that $R_1$ can put at most $m_j$ much mass on each of its paths, since that is how much is added to the stem at the $j$th injury to $R_0$. Thus, the amount that can go into $\del ''$ due to $R_1$ in this phase of the construction is at most $\frac{2m_j}{2^{c_1+1}}=\frac{2m_j}{2^5}$. So the most that $R_1$ can put in during the time when $R_0$ is settling is $\sum\limits_{j}^{k}\frac{m_j}{2^4}\leq \frac{m}{2^4}\leq \frac{1}{2^4}$. After $R_0$ has stopped being injured, injuries to $R_1$ will each put some amount of mass on each of the branches below $n_1$ (as it moves). This will never be lost to any other injuries, so the most that can go into $\del ''$ from $R_1$ in this phase is $\frac{2 \cdot 1}{2^5}=\frac{1}{2^4}$, since 1 is a loose upper bound on the amount that can be put on each path. Thus, the total mass that goes into $\del ''$ due to $R_1$ is less than $\frac{1}{2^3}$. 

The general case for $i > 0$ will follow this pattern. We can imagine the paths in the living tree between $n_l$ and $n_{l+1}$ as reservoirs for mass (the reservoirs are similar to the $Q_{\s}$ defined in the proof of Lemma~\ref{del'}, but depend on the stage of the construction). These paths will change with injuries as the construction goes on, but there are always $2^l$ many. An injury to $R_i$, if it wastes as much mass as possible, takes just the reservoir with the most mass above it and puts that much mass in one of its own reservoirs, then empties all other reservoirs above it (we're only guaranteed that it will take the most massive reservoir). 
As above, mass can also be added to reservoirs with injuries, but any amount that is will be counted in the next injury below it, or end up on the final living tree and so in $\del$. Keeping track of how much mass goes into $\del ''$ based on $R_i$'s injuries now reduces to counting how much mass goes into its reservoirs (the waste will be no more than $\frac{1}{2^{c_i+1}}$ times all the mass that ever shows up in $R_i$'s reservoirs). While $R_0$ is still active, the worst that can happen is that $R_i$ fills each of its reservoirs with some $m_j$, then an injury to $R_{i-1}$ empties them all and puts that $m_j$ into one of its own. Then $R_i$ can fill all it's reservoirs again with the same mass, and $R_{i-1}$ can empty them all to take $m_j$ for one of its own again. This can happen till all of $R_{i-1}$'s reservoirs have $m_j$ in them, and then $R_{i-2}$ can injure it and move that much mass to one of its own reservoir. Then $R_i$ has to fill all its own reservoirs enough times to fill all of $R_{i-1}$'s and then $R_{i-2}$ can do the same thing. This can keep going until $R_1$ has $m_j$ in both of its reservoirs, and then $R_0$ can act and take its $m_j$. Obviously the construction won't actually run in this order, but any actual run will call for less mass from $R_i$. 

This gives us that the most mass the $R_i$ can waste while $R_0$ is active is 

$$\sum\limits_{j=1}^{k}\prod\limits_{p=0}^{i}\frac{m_j 2^{p}}{2^{c_i+1}}=\sum\limits_{j=1}^{k}\frac{m_j 2^{\frac{i^2+i}{2}}}{2^{c_i+1}}\leq \frac{m 2^{\frac{i^2+i}{2}}}{2^{c_i+1}}\leq\frac{2^{\frac{i^2+i}{2}}}{2^{c_i+1}}.$$

After $R_0$ has settled, the worst that can happen is that $R_i$ now needs to provide enough mass in its reservoirs to fill both reservoirs of $R_1$ with up to 1 total mass each (again, an overestimate). This can cost at most 
$$\prod\limits_{p=1}^{i}\frac{1\cdot 2^{p}}{2^{c_i+1}}=\frac{ 2^{\frac{i^2+i}{2}}}{2^{c_i+1}}\leq \frac{ 2^{\frac{i^2+i}{2}}}{2^{c_i+1}}.$$

After this phase it may have to fill all the reservoirs of $R_2$ with up to 1 total mass each, and after that perhaps the same for $R_3$, $R_4$ and so on, till at last it can just fill all its own reservoirs without fear of injury. Thus, the total amount of mass wasted by $R_i$ while all this is happening is bounded by
$$\sum\limits_{l=0}^{i}\prod\limits_{p=l}^{i}\frac{1\cdot 2^p}{2^{c_i+1}}.$$

This is certainly less than 
$$\sum\limits_{l=0}^{i}\prod\limits_{p=0}^{i}\frac{1\cdot2^p}{2^{c_i+1}}\leq\sum\limits_{l=0}^{i}\frac{2^{\frac{i^2+i}{2}}}{2^{c_i+1}}=(i+1)\frac{2^{\frac{i^2+i}{2}}}{2^{c_i+1}}\leq 2^{i+1}\frac{2^{\frac{i^2+i}{2}}}{2^{c_i+1}}=\frac{2^{\frac{i^2+3i+2}{2}}}{2^{c_i+1}}.$$

Recalling that $c_0=0$ and $c_i=4^i$ for $i>1$, it is clear with some calculus that $\frac{2^{\frac{i^2+3i+2}{2}}}{2^{c_i+1}}\leq \frac{1}{2^i}$.

Now we can bound the total mass wasted by all requirements through the whole construction. It is the sum of the mass wasted by each $R_i$, so it is no greater than $\sum\limits_{i=0}^{\infty}\frac{1}{2^i}=2$.

\
\end{proof}

By Lemma~\ref{del'}, $\del ' \leq 2$, so we get $\del \leq 4$, and since $\lam \leq \del$, we find that we do have an \textit{a priori} bound on the measure of the domain of $M_{L}$. It is possibly too big by a factor of $4$, but then the Kraft-Chaitin set that has the length of all descriptions increased by 2 is legitimately a Kraft-Chaitin set, and so we get that $K(\s)\leq^{+}K^{A}(\s)+\fhat(\s)$ for all $\s$ in $\twow$ and all $A \in [T]$, as desired.

This completes the proof that all the infinite paths through $T$ have $K(\s)\leq^+ K^A(\s)+f(\s)$, with the same additive constant. By construction, the set of infinite paths through $T$ is also perfect, so there are uncountably many of them, and so not all are $\Delta^0_2$. 
Also by construction, the only places that two infinite paths through $T$ can differ are at the levels $n_{i}+1$ for all $i$. This means that for any real $A$ there are paths $B$ and $C$ through $T$ with $B(n_i+1)=A(i)$ and $C(n_i+1)=1-A(i)$. Then $B \oplus C$ can compute $A$ by comparing $B$ to $C$ and giving $B$'s value at points where they differ. Additionally, if $A\geq_T 0'$ then $A$ can run the construction and compute which branches will die, so it can compute where the coding locations are and what the paths between them look like. Thus, for any real $A$ there are paths $B$ and $C$ through $T$ such that $B\oplus C \geq _{T} A$, and if $A\geq_T0'$ then $A\equiv _T B\oplus C$. This completes the proof of Theorem~\ref{perfectsetf}.

\begin{proof} (of Theorem~\ref{lowforinfoperfect}) Apply Theorem~\ref{perfectsetf} to the Hirschfeldt-Weber function from Lemma~\ref{hirschwebcor}. \end{proof}

\section{Extensions to Effective Dimension}

Following Hirschfeldt and Weber, we can use different functions $f$ to obtain results similar to Theorem~\ref{lowforinfoperfect} for other lowness properties. First, some definitions.

Effective Hausdorff dimension and effective packing dimension were first defined in terms of martingales succeeding on a given set, but they have equivalent definitions in terms of Komogorov complexity.

\begin{definition} (Mayordomo \cite{hausdorffdef}) The \emph{effective Hausdorff dimension} of a real $S$ is $$\dim(S)=\liminf\limits_{n \rightarrow \infty}\frac{K(S\restrict n)}{n}$$\\
(Athreya, Hitchcock, Lutz, and Mayordomo \cite{packingdef}) The \emph{effective packing dimension} of a real $S$ is $$\Dim(S)=\limsup\limits_{n \rightarrow \infty}\frac{K(S\restrict n)}{n}$$
\end{definition}

These definitions relativize to a real $A$ by taking $K^{A}(S\restrict n)$ instead of $K(S\restrict n)$. We call the relativized versions $\dim^A$ and $\Dim^A$. We can now define a lowness notion for each of these dimensions.

\begin{definition} A real $A$ is \emph{low for effective Hausdorff dimension} if for every real $S$, $\dim(S)=\dim^A(S)$.\\
A real $A$ is \emph{low for effective packing dimension} if for every real $S$, $\Dim(S)=\Dim^A(S)$.\end{definition}

We can use our technique for building perfect sets of reals while controlling their relative Kolmogorov complexities to prove the following theorem.

\begin{thm}\label{dimensionthm} There is a perfect set of reals that are low for effective Hausdorff dimension and effective packing dimension, the class of such reals does not form an ideal, and any real above $0'$ is the join of two such reals\end{thm}

\begin{proof}
Apply Theorem~\ref{perfectsetf} to $f(\s)=\log(|\s|)$. $f$ is certainly total on $\twow$, recursively approximable, and has $\forall i \forall^{\infty} \s \forall s f_s(\s)>i$, so we will get a perfect set of reals $A$ such that $K(\s) \leq^+ K^A(\s)+f(\s)$ for all $\s$. For any real $R$ we have $K^R(\s) \leq^+ K(\s)$, so for the $A$'s we build we get $K(\s)-f(\s) \leq^+K^A(\s)\leq^+K(\s)$. Then for any $S$ and $n$
$$\frac{K(S\restrict n)}{n}-\frac{\log{n}}{n}\leq^+\frac{K^A(S\restrict n)}{n}\leq^+\frac{K(S\restrict n)}{n}.$$
The limit of the $\frac{\log{n}}{n}$ term is $0$ as $n \rightarrow \infty$, so $\dim(S)=\dim^A(S)$ and $\Dim(S)=\Dim^A(S)$.
\end{proof}

Note that the fact that there is a perfect set of such reals was discovered independently by Lempp, Miller, Ng, Turetsky, and Weber and will appear in their forthcoming paper~\cite{perfectdimension}.

\section{Extending to Infinitely Many $f$'s}
Theorem~\ref{perfectsetf} allows us to build a perfect set of reals that all satisfy $K(\s) \leq^{+} K^{A}(\s) +f(\s)$ for $f$'s that have well-behaved recursive approximations. We now want to extend the techniques from Theorem~\ref{perfectsetf} to build such a perfect set that works for all these well-behaved $f$'s. For ease of use we give a definition. 

\begin{definition} A function $f: \twow \rightarrow \mathbb{N}$ is called \textit{finite-to-one approximable} if it has a recursive approximation $f_s$ such that $\forall i \forall^{\infty} \s \forall s \ f_s(\s)>i$. Such an approximation is called a \textit{finite-to-one approximation}.

\end{definition}

We will restate Theorem~\ref{perfectsetallf} using this definition:\\
\begin{thmref} There is a perfect set $\mathcal{P}$ such that for any total $f: \twow \rightarrow \mathbb{N}$ that is finite-to-one approximable, there is a $c_f$ such that for any $A\in \mathcal{P}$ and any $\s \in \twow$, $K(\s) \leq K^A(\s)+f(\s)+c_f$. Moreover, for any real $C$, there exist $A,\ B \in \mathcal{P}$ such that $C \leq_{T} A\oplus B$.
\end{thmref}
What we would like to do is to run a construction like the one used in the proof of Theorem~\ref{perfectsetf} for each finite-to-one approximable $f$  and to weave these all together into one tree. The problem, of course, is that we don't know which functions these are. Thus, we will have to settle for the next best thing: build a tree of runs of the construction for \textit{all} recursive approximations $\phes$ ($\phes(\s)$ is the value given by the $e$th partial recursive function on input $\s$ after $s$ steps of computation, if there is one) with the goal being that the subtree defined by all the correct guesses as to which of these are finite-to-one approximations will be the tree we want.

The construction will be very similar to the one used above, but the branching levels $n_{2e}$ will be used to signal these guesses about the behavior of $\phes$, while the branching levels $n_{2e+1}$ will be used to ensure that the subtree given by a string of guesses remains perfect. Matters will of course be further complicated by trying to achieve $K(\s)\leq ^+ K^{A}(\s)+\phes(\s)$ for all $\s$ and $A$ for all the different $\phes$'s, including those that are wildly poorly-behaved. We will need to have $S_i$ requirements for each $\phes$, and we will need to break the $R_i$ requirements up into requirements that provide a branching node for each path instead of one requirement that branches all paths at a level. We will modify each $\phes$ to $\phat_{e,s}$ analogously to how we defined $\fhat$ above, and use $(c_i)_{i\in\omega}$ as before. The requirements for the construction are as follows.\\

\begin{multline*}R^{\alpha}_i: \text{There is a path through }T  \text{ that follows } \alpha \text{ through }i \text{-many branching } \\ \text{nodes and there is a level where it branches again.}
\end{multline*}
for every $i\in\omega$ and every $\alpha\in\twow$ with $|\alpha|=i$, and\\

$$S^e_{i}: \ \  \text{For all } \s \in \twow \text{ with } \phat_e(\s)=c_i, K(\s)\leq^{+} K^{A}(\s)+\phat_e(\s) \text{ for all } A \in [T]$$

for all $i,e \in \omega$ with $i\geq 2e+1$.

Alongside $T$ we will construct a collection of Kraft-Chaitin sets $(L_e)$, where $L_e$ will be try to witness $K(\s)\leq^{+} K^{A}(\s)+\phat_e(\s)$ for all $\s$ and all $A\in [T]$. We will need to verify that the domains of the $M_{L_e}$ for $e$'s that are finite-to-one approximations all finite. We will say that $S^e_i$ has $e$-\textit{control} of $\s$ at stage $s$ if $\phat_{e,s}(\s)=c_i$. Note that while it will be possible for $\phat_{e,s}(\s)=c_i$ for some $i<2e+1$, $S^e_i$ is not among our requirements, so at stage $s$ (and at all later stages, by the definition of $\phat_{e,s}$) no $S^e$ requirement would have control of $\s$. This ensures, since we only venture a guess as to $\phat_{e,s}$'s behavior at the $2e$th branching level, that we do not allow $\phat_{e,s}$ to injure the parts of the tree below its guessing level. Note that for $\phates$'s that are finite-to-one approximations, the $S^e$ requirements will lose control of only finitely many $\s$, so the $L_e$'s will only fail to contain shorter descriptions of these finitely many.

We will say an $R^{\alpha}_i$ requirement \textit{requires attention} at stage $s$ if it does not have an associated $n^{\alpha}_{i,s}$. $R^{\alpha}_i$ will always act identically to $R^{\beta}_i$ for all $\beta$ with $\beta(2j)=\alpha(2j)$, that is, for paths which have the same guesses as to which of the $\phates$ are finite-to-one approximations. This ensures that the subtrees generated by a sequence of such guesses have certain levels that act as coding locations, as in the earlier construction. 

An $S^e_i$ requirement \textit{requires attention} at stage $s$ if there are a $\s$ that it has $e$-control over and a partial path $\alpha$ in $T_s$ that is alive and such that if $\alpha$ goes through at least $2e$-many branching nodes, then $\alpha(n^{\eta}_{2e})=1$ (where $\eta$ is the appropriate description of $\alpha$'s choices at those nodes) and  $K_{s}^{\alpha}(\s)+\phat_{e,s}(\s)$ is less than the shortest description of $\s$ in $L_{e,s}$. This means that the shorter description of $\s$ is on a path that either has not reached a guessing node for $\phes$ or is guessing that it is a finite-to-one approximation, so we will need to act.

\noindent Now the \textbf{Construction}:\\

\noindent Order the requirements $R^{\empstr}_0, S^0_1, R^{\langle 0 \rangle}_1, R^{\langle 1 \rangle}_1, S^0_2, R^{\langle 00 \rangle}_2, R^{\langle 01 \rangle}_2, R^{\langle 10 \rangle}_2, R^{\langle 11 \rangle}_2, S^0_3, S^1_3, \hdots$.\\

The reader may find it simpler to imagine the requirements being ordered not linearly, but on a binary tree. We can start with the tree $2^{\omega}$ with $R^{\alpha}_{|\alpha|}$ at the node $\alpha$, and then stretch this out by adding instances of the $S^e_i$ requirements where appropriate (i.e. on the branches that would correspond with guesses that $\phes$ is a finite-to-one approximation). The advantage of this approach is that the injury relation is much simpler: an action by a requirement injures the requirements above it in the tree. However, this calls for having multiple instances of each $S^e_i$ requirement and keeping track of which instance is causing an injury, and so in the interests of simpler bookkeeping we have opted to order the requirements linearly with a more complicated injury system. \\

\noindent \textbf{Stage 0}: Set $T_0=\emp$,  $L_{e,0}=\emp$ for every $e$
\\
\\
\textbf{Stage $s$+1}: \\
\textbf{Substage 1}: Calculate $\phat_{e,s+1}(\s)$ and $K_{s+1}^{\alpha}(\s)$ for all living branches $\alpha$ in $T_s$, the first $s+1$ $\s$'s, and $e\leq s+1$. 
\\
\\
For all $\s$ such that $\phat_{e,s+1}(\s) \neq \phat_{e,s}(\s)$ or that $\phat_{e,s+1}(\s)$ first took a value at this stage, if $i\geq 2e+1$ then $S^e_i$ gains $e$-control of $\s$ where $\phat_{e,s+1}(\s)=c_i$ and any other $S^e_j$ loses control of $\s$. Otherwise all $S^e$ requirements lose $e$-control of $\s$. Go to Substage 2:
\\
\\
\noindent\textbf{Substage 2}: Do the following for each of the first $s+1$ requirements that require attention, starting with the one with lowest (closest to 0) priority:
\\
\\
\textbf{Case 1} If it is an $R^{\alpha}_i$ requirement, then all $R^{\beta}_i$, where $\alpha(2j)=\beta(2j)$ for all $j$, will also require attention, so do the following, which will satisfy all of them.\\
\indent 1.) Pick a new $n^{\alpha}_i$ bigger than anything seen yet in the construction. Set $n^{\beta}_i=n^{\alpha}_i$ for all $\beta$ with $\beta(2j)=\alpha(2j)$\\
\indent 2.) Add $\{ \gamma \conc \delta \conc j  | \gamma \text{ is a living leaf node that follows the even bits of } \alpha \linebreak \text{ at the appropriate branching nodes }, \ |\gamma \conc \delta|=n^{\alpha}_i, \ j=0 \text{ or } 1, \text{ and } \delta(k)=0 \linebreak \text{ for all }k \text{ where it is defined} \}$ to $T_s$.\\
\indent 3.) For all $\beta$ with $\beta(2j)=\alpha(2j)$, $n^{\beta}_i$ is now associated to $R^{\beta}_i$ and is a branching level.\\
\\
\\
\textbf{Case 2} If it is an $S^e_i$ requirement, then, since it requires attention, for some partial path $\alpha$ through $T_s$ and some $\s$ that $S^e_i$ has $e$-control over, $K_{s+1}^{\alpha}(\s)+\phat_{e,s+1}(\s)<\min\{l | \langle \s, l \rangle \in L_{e,s}\}$. There are now two subcases, depending on the use $\ua_{s+1} (\tau)$ of the computation $\U_{s+1}^{\alpha}(\tau)=\s$ giving the new shorter description for the lexicographically least $\s$ that this holds for. Let $\eta$ be such that $\alpha$ follows $\eta$ at the branching nodes it goes through. Note that since this causes $S^e_i$ to require attention, it must be the case that if $\ua_{s+1}(\tau)>n^{\eta}_{2e}$ then $\alpha(n^{\eta}_{2e}+1)=1$.
\\
\\
\textbf{Subcase 1}: $\ua_{s+1}(\tau)\leq n^{\eta}_{i}$ (or $n^{\eta}_{i}$ not currently defined) where $\phat_{e,s+1}(\s)=c_i$
\\
As in the simpler case, this is fine. The new description is of a $\s$ with high enough $\phates$ value. So, do the following.\\
\indent 1.) Put a new request $\langle \s , \ K_{s+1}^{\alpha}(\s)+\phat_{e,s+1}(\s)\rangle$ into $L_{e,s}$.\\
\\
\\
\textbf{Subcase 2}: $\ua_{s+1}(\tau)>n^{\eta}_i $  
\\
Again, this case causes a problem. To correct this, do the following.\\
\indent 1.) Injure $R^{\zeta}_i$ for every $\zeta$ with $\zeta(2j)=\eta(2j)$ and run the Injury Subroutine for these requirements.\\
\indent 2.) Let $T_{s+1}=T_s$ and $L_{e,s+1}=L_{e,s}$.\\

Let $T_{s+1}$ and $L_{e, s+1}$ be the new sets acquired by performing all the additions to $T_s$ and $L_{e,s}$ at this stage.

\noindent \textbf{Injury Subroutine} for $\{R^{\zeta}_i|\zeta(2j)=\eta(2j) \}$:

\indent 1.) Find the node $\alpha$ of $T_s$ at level $n^{\eta}_{i}$ ($n^{\eta}_i=n^{\zeta}_i$ for all $\zeta(2j)=\eta(2j)$ by construction: we always work on all these nodes together) and the string $\gamma$ such that $\alpha \conc \gamma$ is a living leaf node of $T_s$ and that maximizes $\sum\limits_{\tau: \ \U_{s+1}^{\gamma}(\tau)\downarrow, \ \U_{s+1}^{\alpha}(\tau)\uparrow} 2^{-|\tau|}$. If there is more than one such pair, choose the leftmost.\\
\indent 2.) For all living nodes $\beta$ at level $n^{\eta}_{i}$ keep the leaf node $\beta \conc \gamma$ alive; kill all other nodes above $\beta$. Set all $R^{\theta}_k$ for $k\geq i$ and $\theta\succeq \zeta$ for some $\zeta(2j)=\eta(2j)$ to requiring attention (i.e. disassociate from each $R^{\theta}_k$ its $n^{\theta}_k$). \\
\\
This ends the construction. We let $T=\bigcup\limits_{s}T_s$ and $L_{e}=\bigcup\limits_{s}L_{e,s}$. What we now need to show is that the subtree given by correctly guessing at each even branching node whether or not $\phates$ is a finite-to-one approximation will be a perfect tree where every path satisfies $K(\s)\leq^{+}K^{A}(\s)+\phat_e(\s)$ for all $\s$ and a constant that depends only on $e$. First, we show that the requirements involved in this subtree are all satisfied.\\

\begin{lemma}\label{finiteinjuryrightes}If $\alpha\in\twow$ is such that $\alpha(2e)=1$ if and only if $\phates$ is a finite-to-one approximation, then $R^{\alpha}_{|\alpha|}$ is injured only finitely often\end{lemma}

\begin{proof} The proof is similar to the proof of Lemma~\ref{finiteinjury}. Let $\alpha$ be as in the statement of the lemma. Then the only injuries to $R^{\alpha}_{|\alpha|}$ can result from being in Subcase 2 for some $S^e_i$ requirement with $i<|\alpha|$. There are finitely many such requirements. In order for an injury to occur, a new description of some $\s$ that $S^e_i$ has $e$-control over must converge on a path $\gamma$ through $T_s$ with a use at least as big as $n^{\eta}_{2e+1}$ and such that $\gamma(n^{\eta}_{2e}+1)=1$ (so $\eta(2e)=1$) for some $\eta$ that describes $\gamma$'s branchings through $T_s$ and agrees with $\alpha$ on the even bits(or it wouldn't cause an injury to $R^{\alpha}_{|\alpha|}$. Since $\alpha$ and $\eta$ agree on the even bits of $\gamma$'s branchings and $\gamma$ goes through the branch that guesses that $\phates$ is a finite-to-one approximation, $\phates$ must actually be a finite-to-one approximation by the assumption on $\alpha$. Thus, each requirement that could injure $R^{\alpha}_{|\alpha|}$ will only ever have $e$-control over finitely many $\s$. If we view the paths that agree with $\alpha$ on the even bits (i.e. have exactly the same guesses about the $\phates$, so are injured at exactly the same stages) as a subtree, the lemma follows by the same method as Lemma~\ref{finiteinjury}.

\end{proof}

\begin{lemma}\label{rightonessatisfied} The $S^e_i$ requirements where $\phates$ is a finite-to-one approximation and the $R^{\alpha}_{|\alpha|}$ requirements where $\alpha$ is such that $\alpha(2d)=1$ if and only if $\phat_{d,s}$ is a finite-to-one approximation are all eventually satisfied. \end{lemma}

\begin{proof}
Note that every requirement that requires attention at the beginning of a stage is allowed to act during that stage. Some requirements will require attention infinitely often (either because they will have $e$-control of infinitely may $\s$'s or because they will be injured infinitely often), but these requirements can only act on parts of $T_s$ that are associated with incorrect guesses about the $\phates$'s. Thus, their actions will not interfere with the actions of the requirements that satisfy the hypotheses of the lemma.
 Since $e$ is such that $\phates$ is a finite-to-one approximation, there are only finitely many $\s$ that $S^e_i$ will ever have $e$-control over. By Lemma~\ref{finiteinjuryrightes}  any $\alpha$ with $\alpha(2d)=1$ if and only if $d$ is a finite-to-one approximation must only be injured finitely often, and this includes $|\alpha|>2e$. Since these $\alpha$'s can be injured only finitely often, $S^e_i$ can cause only finitely many injuries (any injury caused by $S^e_i$ is an injury to some $R^{\beta}_i$). Thus, eventually $S^e_i$ will only be in Subcase 1 when it requires attention. Each time it acts after this point, the length of description of one of the finitely many $\s$ that will cause it to require attention again drops by at least 1. This can only happen finitely often, after which point $S^e_i$ is satisfied.

$R^{\alpha}_{|\alpha|}$ can only be injured finitely often, so eventually it reaches a stage after which it is never injured again. Once it acts after that stage it is satisfied.
\end{proof}

Let $T^{*}$ be the subtree of $T$ defined by

\begin{multline*}T^{*}=\{ \tau \in T | \forall e \text{ if } \alpha \text{ is the } 2e \text{th branching node with } \alpha \prec \tau \text{, then } \alpha \conc 1 \prec \tau \text{ if} \\ \text{ and only if } \phates \text{ is a finite-to-one approximation} \}.
\end{multline*}

That is, $T^*$ is the subtree given by guessing correctly at every even branching node whether or not $\phates$ is a finite-to-one approximation. Both extensions through any odd branching node are in $T^*$, and so with Lemmas~\ref{finiteinjuryrightes} and \ref{rightonessatisfied} we get that $T^*$ is a perfect tree. The collection of paths through  $T^*$, $[T^*]$, is the class we are looking for. All that remains is to show that for any finite-to-one approximation $\phates$, the domain of $M_{L_e}$ is finite. This will actually be true for any $e$, so every $L_e$ will be a legitimate Kraft-Chatin set, but when $\phates$ is not a finite-to-one approximation $L_e$ might just be witnessing that some degenerate path through $T$ is low for $K$, or might fail to have descriptions for infinitely many $\s$. We can define $\lam_e$, $\del_e$, $\del'_e$, and $\del''_e$ analogously to the definitions in Section 4, and we now bound the domain of $M_{L_e}$ as before.

\begin{lemma} For any $e$,  $\del'_e \leq 2$ \end{lemma}
\begin{proof}
The proof is exactly like the proof of Lemma~\ref{del'} with the exception that some $\alpha_\s$'s will not be defined on the living subtree of $T$, since infinite injuries to branches off of $T^*$ can cause some paths to be isolated. For these $\s$, let $m_\s=0$ and the calculations go through as before (This will in general be an over-estimate, since we will not actually put requests into $L_e$ for $\s$ with $\phat_e(\s)\leq c_{2e}$). 

\end{proof}

\begin{lemma} For any injury to requirement $R^{\alpha}_i$ at stage $s$ in the construction, the amount that is paid into $\del_e$ on the paths above the branching nodes that are affected (those that are above paths that follow $\eta$ through lower branching nodes for some $\eta$ with $\eta(2j)=\alpha(2j)$) is no more than $\frac{1}{2}\cdot\frac{1}{2^{c_i}}$ times the mass $m$ that converges on the path chosen by the run of Injury Subroutine for this injury.

\end{lemma}
\begin{proof}
Again, the proof is the same as the proof of Lemma~\ref{injurybound}. In fact, there will be fewer nodes above which waste occurs, because only those that are along paths that guess that the $\phat_j$ that caused the injury is a finite-to-one approximation will generate waste.
\end{proof}

\begin{lemma}For any $e$, $\del''_e \leq 2$
\end{lemma}

\begin{proof}
Here, again, the method from the proof of the earlier Lemma~\ref{del''} goes through without alteration. Nothing in the argument relied on there being only finitely many injuries that filled a given reservoir. Even if a reservoir gets mass added by infinitely many injuries, the amount of mass that can go through it is bounded as before. 
\end{proof}

This ends the proof of Theorem~\ref{perfectsetallf}. 

\bibliographystyle{acm}
\bibliography{lowforinfo}{}

\end{document}